\documentclass[12pt]{article}
\usepackage{amsmath,amsthm,ascmac,amssymb,enumerate,graphics,enumitem,multicol,cancel}
\usepackage{color}
\usepackage{ulem}

\usepackage[top=20truemm,bottom=25truemm,left=25truemm,right=25truemm]{geometry}

\theoremstyle{break}
\newtheorem{thm}{TeiRi}[section]

\newtheorem{Prop}[thm]{Proposition}
\newtheorem{Lem}[thm]{Lemma}
\newtheorem{Thm}[thm]{Theorem}
\newtheorem{Cor}[thm]{Corollary}

\newtheorem{Exa}[thm]{Example}

\newcommand{\Or}{\mathcal{O}}

\newcommand{\afrak}{\mathfrak{a}}

\newcommand{\bfrak}{\mathfrak{b}}
\newcommand{\cfrak}{\mathfrak{c}}

\newcommand{\pfrak}{\mathfrak{p}}

\newcommand{\Z}{\mathbb{Z}}
\newcommand{\Q}{\mathbb{Q}}

\newcommand{\maru}[1]{\raise0.2ex\hbox{\textcircled{\scriptsize{#1}}}}

\title{Integers of the Form $ax^2+bxy+cy^2$}
\author{Naoki Uchida\thanks{Tokyo University of Science.}
}
\date{\today}

\makeatletter
\newcommand{\subjclass}[2][2010]{%
  \let\@oldtitle\@title%
  \gdef\@title{\@oldtitle\footnotetext{#1 \emph{Mathematics subject classification$(s)$.} #2}}%
}
\newcommand{\keywords}[1]{%
  \let\@@oldtitle\@title%
  \gdef\@title{\@@oldtitle\footnotetext{\emph{Key words and phrases.} #1.}}%
}
\makeatother
\subjclass{11E16, 
11E25
, 
11N32
, 
11R65.
} 

\keywords{binary quadratic forms, 
representation of integers by quadratic forms, 
Diophantine equations, 
orders in 
quadratic fields, 
decomposition of ideals
} 

\begin{document}

\maketitle

\begin{abstract}
We provide a characterization of integers represented by the positive definite binary quadratic form $ax^2+bxy+cy^2$. In order to prove the main theorem, we define the ``relative conductor'' of two orders in an imaginary quadratic field. Then we provide a characterization of decomposition of proper ideals of orders in imaginary quadratic fields. Moreover we present some interesting examples of the main theorem. 
\end{abstract}

\section{Introduction}%
Fermat stated that for a prime $p$, the following holds.
\[
\text{$p=x^2+y^2$ has an integer solution 
$\iff$ 
$p=2$ or $p\equiv1 \bmod4$.}
\]
Many mathematicians, including 
Lagrange, Legendre and Gauss, developed genus theory and the theory of composition 
of quadratic forms. 
These theories enable us to prove that 
for a prime $p$, the following holds.
\begin{align*}
\text{$p=3x^2+2xy+3y^2$ has an integer solution 
$\iff$ 
$p\equiv3 \bmod8$.} 
\end{align*}
Gauss's published works mention cubic reciprocity \cite{Gau}. 
This implies 
that for a prime $p$, the following holds.
\begin{align*}
&\text{$p=4x^2+2xy+7y^2$ has an integer solution} \\
\iff 
&\text{$p\equiv1 \bmod3$ and $2$ is not a cubic residue modulo $p$.} 
\end{align*}
Cox stated a characterization of primes of the form $x^2+ny^2$, where $n$ is a fixed positive integer \cite{Cox}. 

However, it is natural to consider 
what happens if we replace a prime $p$ with an arbitrary integer
. 
There are some previous studies about it. 
For example, 
Fermat knew 
an equivalent condition for when 
integers can be written as $x^2+y^2$, where $x, y$ are some integers \cite[Chapter 2]{Gro}. 
Koo and Shin gave an equivalent condition for when the equation $pq=x^2+ny^2$ has an integer solution, where $n$ is a fixed positive integer, and $p, q$ are distinct odd primes not dividing $n$ \cite{Koo}. 
Cho presented a characterization of integers, relatively prime to 
$2nm$, 
represented by the form $x^2+ny^2$
 with $x\equiv1\bmod m$, $y\equiv0\bmod m$, where $m, n$ are fixed positive integers \cite{Cho}. 
He also presented a characterization of integers, relatively prime to 
$2(1-4n)m$, 
represented by the form $x^2+xy+ny^2$
 with $x\equiv1\bmod m$, $y\equiv0\bmod m$, where $m, n$ are fixed positive integers \cite{Cho2}. 

However, we present a characterization of 
integers, not necessarily prime to the discriminant $D=b^2-4ac$, represented by the positive definite binary quadratic form $ax^2+bxy+cy^2$, where $a, b, c\in\Z$ 
(see Theorem \ref{MainTheorem}). 
This is the main theorem in this paper. 

The main theorem leads to the following example which Fermat knew. 

\begin{Exa}\label{exa-4}
Let $m$ be an arbitrary positive integer. 
Write 
\[
m=p_1\cdots p_r \cdot {q_1}^{e_1}\cdots{q_s}^{e_s},
\]
where
\begin{itemize}
\item the $p_i$'s are primes with $p_i=2$ or $p_i\equiv1 \bmod4$, 
\item the $q_j$'s are distinct primes with $q_j\equiv3 \bmod4$, 
\item $r, s\geq 0$, $e_j >0$.
\end{itemize}
Then, the followings are equivalent.
\begin{enumerate}[label=$(\mathrm{\roman*})$]
\item $m=x^2+y^2$ has an integer solution.
\item All $e_j$'s are even. 
\end{enumerate}
\end{Exa}

Moreover the main theorem leads to the following interesting examples.

\begin{Exa}\label{exa-32}
Let $m$ be an arbitrary positive integer. 
Write 
\[
m=p_1\cdots p_r \cdot {q_1}^{e_1}\cdots{q_s}^{e_s} \cdot {2}^{h},
\]
where
\begin{itemize}
\item the $p_i$'s are primes with $p_i\equiv1, 3 \bmod8$, 
\item the $q_j$'s are distinct primes with $q_j\equiv5, 7 \bmod8$. 
\item $r, s\geq 0$, $e_j >0$, $h \geq0$.
\end{itemize}
Then, the followings are equivalent.
\begin{enumerate}[label=$(\mathrm{\roman*})$]
\item $m=3x^2+2xy+3y^2$ has an integer solution.
\item All $e_j$'s are even, and 
one of the following holds.
\begin{enumerate}[label=$(\mathrm{\alph*})$]
\item the number of primes $p_k$'s with $p_k\equiv3 \bmod8$ is odd, $h=0$, 
\item $h\geq2$. 
\end{enumerate}
\end{enumerate}
\end{Exa}

Note that we do not assume that $m$ is relatively prime to $2$. 
\begin{Exa}\label{exa-108}
Let $m$ be an arbitrary positive integer. 
Write 
\[
m=p_1\cdots p_r \cdot {q_1}^{e_1}\cdots{q_s}^{e_s} \cdot {2}^{h_2} {3}^{h_3}, 
\]
where 
\begin{itemize}
\item the $p_i$'s are primes with $p_i\equiv1 \bmod3$, 
\item the $q_j$'s are distinct primes with $q_j\neq2$ and $q_j\equiv2 \bmod3$, 
\item $r, s\geq 0$, $e_j >0$, $h_2 \geq0$, $h_3 \geq0$. 
\end{itemize}
Then, the followings are equivalent.
\begin{enumerate}[label=$(\mathrm{\roman*})$]
\item $m=4x^2+2xy+7y^2$ has an integer solution.
\item All $e_j$'s are even, and 
one of the following holds.
\begin{enumerate}[label=$(\mathrm{\alph*})$]
\item 
There exists at least one $p_k$ satisfying 
that $2$ is not a cubic residue 
modulo $p_k$, $h_2=h_3=0$, 
\item $h_2$ is even and $h_3\neq1$, except for $(h_2,h_3)=(0,0)$. 
\end{enumerate}
\end{enumerate}
\end{Exa}

Note that we do not assume that $m$ is relatively prime to $2$ or $3$. 
Moreover, the form is not a principle form $x^2+ny^2$. 

The outline of this paper is as follows. 
In Section \ref{Themainth}, we state the main theorem in this paper, and explain the idea of 
its proof 
. 
In Section \ref{Orders}, we prepare some propositions for orders in imaginary quadratic fields. 
In Section \ref{DecompOk}, we recall a characterization of decomposition of ideals of imaginary quadratic fields. 
In Section \ref{Relative}, we define the ``relative conductor'' of two orders in an imaginary quadratic field. 
Using the relative conductor, we provide a characterization of decomposition of proper ideals of orders in imaginary quadratic fields in Section \ref{Decomp}. 
In Section \ref{ProofoftheM}, we derive the main theorem by using the above characterization. 
In Section \ref{ExamplesoftheM}, we see some examples of the main theorem, including Examples \ref{exa-4}, \ref{exa-32} and \ref{exa-108}. 
Furthermore, we provide characterizations of prime powers $l^h$, where $l$ divide the conducter, represented by the positive definite binary quadratic form $ax^2+bxy+cy^2$.


\section{The main theorem}\label{Themainth}
Let $f(x,y)=ax^2+bxy+cy^2$ be an integral binary quadratic form (for short, a form).
It is said that $f(x,y)$ is {\it{primitive}} if its coefficients $a,b$ and $c$ are relatively prime. 
Note that any form is an integer multiple of a primitive form, 
thus we will exclusively consider 
primitive forms. 
An integer $m$ is {\it{represented}} by a form $f(x,y)$ if the equation 
\[
m=f(x,y)
\]
has an integer solution in $x$ and $y$.
If such $x$ and $y$ are relatively prime, we say that $m$ is {\it{properly represented}} by $f(x,y)$.

We say that two forms $f(x,y)$ and $g(x,y)$ are {\it{properly equivalent}} if there is an element 
$
\begin{pmatrix}
 p & q \\
 r & s 
\end{pmatrix}
$
$\in \mathrm{SL}(2,\Z)$ such that
\[
f(x,y)=g(px+qy,rx+sy).
\]
The 
proper equivalence of forms 
is an 
equivalence relations.
An important thing is that 
properly 
equivalent forms represent 
the same 
numbers, 
and the same is true for proper representations. 
Note also that any form properly equivalent to a primitive form is itself primitive.

We define the discriminant $D$ of $ax^2+bxy+cy^2$ to be $D=b^2-4ac$.
Note that $D\equiv0, 1\bmod4$, and 
properly equivalent forms have the same discriminant. 

We say that a form $f(x,y)$ is {\it{positive definite}} 
if 
$f(x,y)$ represents only positive integers 
when $x,y\neq 0$. 
Any positive definite form has a negative discriminant. 
From this point, we will specialize to the primitive positive definite case. 

We denote by $C(D)$ the set of proper equivalence classes 
of positive definite forms of discriminant $D$.
Then the Dirichlet composition induces a well-defined binary operation on $C(D)$ which makes  $C(D)$ into a finite Abelian group (see Cox \cite[Theorem 3.9]{Cox}). 
We say $C(D)$ is the {\it{form class group}}.

Now we prepare some notations from algebraic number theory. 
For a 
quadratic 
field $K$, 
we denote by $\Or_K$, $d_k$ and $I_K$ the ring of integers of $K$, the discriminant of $K$ and 
the group of all fractional ideals of $K$, 
respectively.

Now we state the main theorem in this paper. 
\begin{Thm}\label{MainTheorem}
Let $ax^2+bxy+cy^2$ be a primitive positive definite form of discriminant $D$. 
Suppose that $K=\Q(\sqrt{D})$, $f=\sqrt{D/d_K}$. 
Write 
$f={l_1}^{\lambda_1}\cdots{l_t}^{\lambda_t}$, 
where 
$t\geq 0$, $\lambda_k >0$, and 
the $l_k$'s are distinct primes.
Let $m$ be an arbitrary positive integer. 
Write 
\[
m=p_1\cdots p_r \cdot {q_1}^{e_1}\cdots{q_s}^{e_s} \cdot {l_1}^{h_1}\cdots{l_t}^{h_t}, 
\]
where 
\begin{itemize}
\item the $p_i$'s are primes relatively prime to $f$ with $(D/p_i)=0,1$, 
\item the $q_j$'s are distinct primes relatively prime to $f$ with $(D/q_j)=-1$, 
\item $r, s\geq 0$, $e_j >0$, $h_k \geq0$. 
\end{itemize}
Note that $(D/p)$ is the Legendre symbol. 
Then, the followings are equivalent.
\begin{enumerate}[label=$(\mathrm{\roman*})$]
\item $m=ax^2+bxy+cy^2$ has an integer solution.
\item 
All $e_j$'s are even, 
and 
there exist 
\begin{itemize}
\item primitive positive definite forms $f_i(x,y)$'s of discriminant $D$ representing $p_i$,
\item primitive positive definite forms $g_k(x,y)$'s of discriminant $D$ representing ${l_k}^{h_k}$, 
\end{itemize}
such that 
\[
\text{$[ax^2+bxy+cy^2]=[f_1(x,y)]\cdots[f_r(x,y)]\cdot[g_1(x,y)]\cdots[g_t(x,y)]$ in $C(D)$}. 
\]
\end{enumerate}
\end{Thm}

The proof of Theorem \ref{MainTheorem}-[(ii)$\Rightarrow$(i)] is relatively easy. 
In order to prove this, we prepare the following lemmas.

\begin{Lem}\label{Dircomp}
Let $f(x,y)$, $g(x,y)$, $h(x,y)$ be primitive positive definite forms of discriminant $D$. 
Assume that 
$f(x,y)$, $g(x,y)$ represent integers $m$,$n$, respectively, 
and $[f(x,y)][g(x,y)]=[h(x,y)]$ in $C(D)$. 
Then $h(x,y)$ represents $mn$. 
\end{Lem}
\begin{proof}
This follows immediately by definition of composition. 
\end{proof}

\begin{Lem}\label{mn^2}
Let $f(x,y)$ be a primitive positive definite form of discriminant $D$, 
and let $n$ be an integer. 
If $f(x,y)$ represents an integer $m$, then $f(x,y)$ represents $mn^2$. 
\end{Lem}
\begin{proof}
This follows immediately by the form $f(x,y)$. 
\end{proof}

Theorem \ref{MainTheorem}-[(ii)$\Rightarrow$(i)] follows immediately by Lemmas \ref{Dircomp} and \ref{mn^2}. 
However, the proof of Theorem \ref{MainTheorem}-[(i)$\Rightarrow$(ii)] is 
difficult. Now we explain the idea of the proof. 
Since the ring of integers $\Or_K$ is a Dedekind domain, 
we obtain the following theorem. 

\begin{Thm}\label{result0}
Let $\Or_K$ be the ring of integers in 
an imaginary quadratic field $K$. 
Let $\afrak$ be an $\Or_K$-ideal, 
let $m$ be an arbitrary positive integer. 
Write 
\[
m=p_1\cdots p_r \cdot {q_1}^{e_1}\cdots{q_s}^{e_s}, 
\]
where 
\begin{itemize}
\item the $p_i$'s are primes with $(d_K/p_i)=0,1$, 
\item the $q_j$'s are distinct primes with $(d_K/q_j)=-1$, 
\item $r, s\geq 0$, $e_j >0$. 
\end{itemize}
Then, the followings are equivalent.
\begin{enumerate}[label=$(\mathrm{\roman*})$]
\item $N(\afrak)=m$.
\item All $e_j$'s are even, 
and 
there exist 
prime $\Or_K$-ideals $\pfrak_i$'s where $N(\pfrak_i)=p_i$ 
such that 
\[
\afrak=\pfrak_1\cdots\pfrak_r \cdot (q_1\Or_K)^{e_1/2}\cdots (q_s\Or_K)^{e_s/2}.
\]
\end{enumerate}
\end{Thm}

We prepare some terminologies for orders in quadratic fields. 
Let $\Or$ be an order of discriminant $D$ in a quadratic field. 
The index $f=|\Or_K/\Or|$ is called the {\it{conductor}} of the order. 
Then one can show that $D=f^2d_K$, namely $f=\sqrt{D/d_K}$ 
(see Cox \cite[\S7-A]{Cox}). 
A fractional $\Or$-ideal $\afrak$ is {\it{proper}} 
provided that 
\[
\Or = \{\beta\in K \mid \beta\afrak \subset \afrak\}.
\] 
A fractional $\Or$-ideal $\afrak$ is {\it{invertible}} 
if there exists a fractional ideal $\bfrak$ such that $\afrak\bfrak=\Or$. 
The notions of 
properness and invertibility coincide (see Cox \cite[Proposition 2.1]{Cox}).

By applying 
Proposition \ref{Ikf,Iof} to Theorem \ref{result0}, we get the following corollary.

\begin{Cor}\label{result1}
Let $\Or$ be the order of conductor $f$ in an imaginary quadratic field $K$, 
and let $D$ be the discriminant of $\Or$.
Let $\afrak$ be a proper $\Or$-ideal, 
and 
let $m$ be an arbitrary positive integer relatively prime to $f$. 
Write 
\[
m=p_1\cdots p_r \cdot {q_1}^{e_1}\cdots{q_s}^{e_s}, 
\]
where 
\begin{itemize}
\item the $p_i$'s are primes with $(D/p_i)=0,1$, 
\item the $q_j$'s are distinct primes with $(D/q_j)=-1$, 
\item $r, s\geq 0$, $e_j >0$. 
\end{itemize}
Then, the followings are equivalent.
\begin{enumerate}[label=$(\mathrm{\roman*})$]
\item $N(\afrak)=m$.
\item All $e_j$'s are even, and 
there exist prime $\Or$-ideals $\pfrak_i$'s where $N(\pfrak_i)=p_i$ 
such that 
\[
\afrak=\pfrak_1\cdots\pfrak_r (q_1\Or)^{e_1/2}\cdots (q_s\Or)^{e_s/2}.
\]
\end{enumerate}
\end{Cor}

We provide a proof of Theorem \ref{result0} and Corollary \ref{result1} in Section \ref{DecompOk}, 
although 
these seem to be well-known for experts. 
Furthermore, 
Corollary \ref{result1} is not enough for the main theorem, 
since we assumed that $m$ is relatively prime to the conductor $f$. 
Thus we need the following characterization which we prove in Section \ref{Decomp}.

\begin{Thm}\label{yosoCor}
Let $\Or$ be the order of conductor 
$f$ 
in an imaginary quadratic field $K$.  
Write 
$f={l_1}^{\lambda_1}\cdots{l_t}^{\lambda_t}$, 
where 
$t\geq 0$, $\lambda_k >0$, and 
the $l_k$'s are distinct primes. 
Let 
$D$ be the discriminant of $\Or$.
Let $\afrak$ be a proper $\Or$-ideal, 
and 
let $m$ be an arbitrary positive integer. 
Write 
\[
m=p_1\cdots p_r \cdot {q_1}^{e_1}\cdots{q_s}^{e_s} \cdot {l_1}^{h_1}\cdots{l_t}^{h_t}, 
\]
where 
\begin{itemize}
\item the $p_i$'s are primes relatively prime to $f$ with $(D/p_i)=0,1$, 
\item the $q_j$'s are distinct primes relatively prime to $f$ with $(D/q_j)=-1$, 
\item $r, s\geq 0$, $e_j >0$, $h_k \geq0$. 
\end{itemize}
The followings are equivalent.
\begin{enumerate}[label=$(\mathrm{\roman*})$]
\item $N(\afrak)=m$.
\item All $e_j$'s are even, and 
there exist 
\begin{itemize}
\item prime $\Or$-ideals $\pfrak_i$'s where $N(\pfrak_i)=p_i$, 
\item proper $\Or$-ideals $\cfrak_k$'s where $N(\cfrak_k)={l_k}^{h_k}$, 
\end{itemize}
such that 
\[
\afrak=\pfrak_1\cdots\pfrak_r \cdot (q_1\Or)^{e_1/2}\cdots (q_s\Or)^{e_s/2} 
\cdot \cfrak_1 \cdots \cfrak_t.
\]
\end{enumerate}
\end{Thm}

Note that 
to prove this theorem, we define the relative conductors of orders in Section \ref{Relative}. 

Finally, 
using the fact that 
the form class group $C(D)$ is isomorphic to the ideal class group $C(\Or)$ (see Proposition \ref{C(D),C(O)}), we 
derive Theorem \ref{MainTheorem}-[(i)$\Rightarrow$(ii)] in Section \ref{ProofoftheM}.


\section{Orders in imaginary quadratic fields}\label{Orders}

We prepare some propositions for orders. 
\begin{Prop}\label{normseki}
Let $\Or$ be an order 
in an imaginary quadratic field. Then
\begin{enumerate}[label=$(\mathrm{\roman*})$]
\item $N(\afrak\bfrak)=N(\afrak)N(\bfrak)$ 
for proper $\Or$-ideals $\afrak$ and $\bfrak$.
\item $\afrak\overline{\afrak}=N(\afrak)\Or$ for proper $\Or$-ideal $\afrak$.
\end{enumerate}
\end{Prop}
\begin{proof}
See Cox \cite[Lemma 7.14]{Cox}.
\end{proof}

For an order $\Or$ in a quadratic field $K$, 
we denote by $I(\Or)$ and $P(\Or)$ the group of proper fractional $\Or$-ideals and the subgroup of principal fractional ideals, respectively. 
The quotient $C(\Or)=I(\Or)/P(\Or)$ is the {\it{ideal class group}} of the order $\Or$. 
Let $m$ be a positive integer. 
We say that 
a proper $\Or$-ideal $\afrak$ is 
{\it{prime to}} $m$ provided that $\afrak+m\Or = \Or$. 
An $\Or$-ideal $\afrak$ is prime to $m$ 
if and only if 
its norm $N(\afrak)$ is relatively prime to $m$ (see Cox \cite[\S7-C]{Cox}). 

Let $\Or$ be an order of conductor $f$ in an imaginary quadratic field $K$. 
We denote by $I(\Or,f)$ and $I_K(f)$ 
the subgroup of $I(\Or)$ generated by $\Or$-ideals prime to $f$ and 
the subgroup of $I_K$ generated by $\Or_K$-ideals prime to $f$, respectively. 
\begin{Prop}\label{Ikf,Iof}
Let $\Or$ be the order of conductor $f$ in an imaginary quadratic field $K$. Then 
the map $\afrak\mapsto\afrak\cap\Or$ induces 
an isomorphism $I_K(f)\simeq I(\Or,f)$, 
and the inverse of this map is given by $\afrak\mapsto\afrak\Or_K$. 
\end{Prop}
\begin{proof}
See Cox \cite[Proposition 7.20]{Cox}.
\end{proof}

We relate the form class group $C(D)$ to the ideal class group $C(\Or)$. 

\begin{Prop}\label{existO}
Any negative integer $D\equiv0, 1\bmod4$ is the discriminant of an order $\Or$ in an imaginary quadratic field $K$. Furthermore, $D$ determines $\Or$ uniquely, and $K=\Q(\sqrt{D})$. 
\end{Prop}
\begin{proof}
See Cox \cite[\S7-A]{Cox}.
\end{proof}

\begin{Prop}\label{C(D),C(O)}
Let $\Or$ be 
an 
order of discriminant $D$ in an imaginary quadratic field. 
Then
\begin{enumerate}[label=$({\rm{\roman*}})$]
\item If $f(x,y)=ax^2+bxy+cy^2$ is a form of discriminant $D$, 
then $\langle a,(-b+\sqrt{D})/2 \rangle_\Z$ is a proper ideal of $\Or$.
\item The map sending $f(x,y)$ to $\langle a,(-b+\sqrt{D})/2 \rangle_\Z$ induces an isomorphism between $C(D)$ and $C(\Or)$.
\item A positive integer $m$ is represented by a form $f(x,y)$ if and only if $m$ is the norm $N(\mathfrak{a})$ of some ideal $\mathfrak{a}$ in the corresponding ideal class in $C(\Or)$ 
.
\end{enumerate} 
Here, we set 
$\langle a,(-b+\sqrt{D})/2 \rangle_\Z$ 
$=$ 
$\{ma+n(-b+\sqrt{D})/2 \mid m,n\in\Z \}$.
\end{Prop}
\begin{proof}
See Cox \cite[Theorem 7.7]{Cox}.
\end{proof}


\section{Decomposition of $\Or_K$-ideals}\label{DecompOk}
In this section, we prove Theorem \ref{result0} and Corollary \ref{result1}, 
although it may be well-known for experts. 
\begin{proof}[Proof of Theorem \rm{\ref{result0}}]
$[(\mathrm{i})\Rightarrow(\mathrm{ii})]$ 

\noindent
Since $\Or_K$ is a Dedekind domain, 
$\afrak$ can be written as a product 
\[
\afrak=
\pfrak_1\cdots\pfrak_{r} \cdot (q_1\Or_K)^{f_1}\cdots (q_{s}\Or_K)^{f_{s}}
\]
of prime ideals, 
where 
$N(\pfrak_i)=p_i$, $f_i>0$, 
and the $p_i$'s split or ramify 
and the $q_j$'s are inert. 

Since the norm of $\Or_K$-ideals preserves multiplication, we have
\begin{align*}
p_1\cdots p_r \cdot {q_1}^{e_1}\cdots{q_s}^{e_s}
&=m \\
&=N(\afrak) \\
&=N(\pfrak_1)\cdots N(\pfrak_{r})\cdot N(q_1\Or_K)^{f_1}\cdots N(q_{s}\Or_K)^{f_{s}} \\
&=p_1\cdots p_r \cdot {q_1}^{2f_1}\cdots{q_{s}}^{2f_{s}}. 
\end{align*}
Thus 
we see 
$e_j=2f_j$. 
Therefore all $e_j$'s are even. 
This completes the proof.

\noindent
$[(\mathrm{ii})\Rightarrow(\mathrm{i})]$ 

\noindent
This follows immediately, since norm of $\Or_K$-ideals preserves multiplication. 
\end{proof}

\begin{proof}[Proof of Corollary \rm{\ref{result1}}]
Since $D=f^2 d_K$, we see that 
$(D/p_i)
=(d_K/p_i)$, 
$(D/q_j)
=(d_K/q_j)$. 
Thus 
Proposition \ref{Ikf,Iof} and 
Theorem \ref{result0} 
imply the corollary.
\end{proof}

%

\section{Relative conductors of orders}\label{Relative}
We define the relative conductors of orders in this section, 
in order to prove Theorem \ref{yosoCor}, 
which is a characterization of decomposition of proper $\Or$-ideals not relatively prime to the conductor $f$. 
Let $\Or$, $\Or'$ be orders 
in an imaginary quadratic field with $\Or \subset\Or'$. 
We say that the index $r=|\Or'/\Or|$ is the 
{\it{relative conductor}} of $\Or$ in $\Or'$. 

\begin{Prop}\label{seki}
Let $\Or$, $\Or'$ be orders in an imaginary quadratic field with $\Or\subset\Or'$. 
Let $\afrak,\bfrak$ be 
$\Or$-ideals. Then
\[
(\afrak\bfrak)\Or' = (\afrak\Or') (\bfrak\Or').
\]
\end{Prop}
\begin{proof}
This follows immediately by definition of orders. 
\end{proof}

\begin{Prop}\label{proper}
Let $\Or$, $\Or'$ be orders in an imaginary quadratic field with $\Or\subset\Or'$. 
If $\afrak$ is a proper $\Or$-ideal, then 
$\afrak\Or'$ is also 
a proper $\Or'$-ideal.
\end{Prop}
\begin{proof}
Proposition \ref{normseki} %
implies that 
$\afrak\overline{\afrak}=N(\afrak)\Or$. 
By Proposition \ref{seki}, we have
\[
(\afrak\Or')(\overline{\afrak\Or'})
=N(\afrak)\Or'.
\]
Thus we get
\[
(\afrak\Or')\cdot1/N(\afrak)(\overline{\afrak\Or'})=\Or'.
\]
Therefore $\afrak\Or'$ is invertible. 
Hence 
we see that $\afrak\Or'$ is proper. 
\end{proof}

\begin{Prop}\label{norm}
Let $\Or$, $\Or'$ be orders in an imaginary quadratic field with $\Or\subset\Or'$. 
Let $\afrak$ be a proper $\Or$-ideal. Then
\[
N(\afrak) = N(\afrak\Or') 
\]
\end{Prop}
\begin{proof}
Proposition \ref{normseki} implies that $\afrak\overline{\afrak}=N(\afrak)\Or$. 
By Proposition \ref{seki}, we have
\[
\afrak\Or' \overline{\afrak\Or'} 
= N(\afrak)\Or'. 
\]
Since $\afrak\Or'$ is proper by Proposition \ref{proper}, 
we have 
$N(\afrak) = N(\afrak\Or')$. 
\end{proof}

Now we pay attention to 
whether $\Or'$-ideals will be prime to the relative conductor $r$.

\begin{Prop}\label{sonaranorm}
Let $\Or$, $\Or'$ be orders in an imaginary quadratic field with $\Or\subset\Or'$, 
and 
$r$ 
the relative conductor of $\Or$ in $\Or'$. 
Let $\afrak'$ be an 
$\Or'$-ideal prime to $r$. Then
\begin{align*}
N(\afrak'\cap\Or)=N(\afrak'), 
\end{align*}
thus $\afrak'\cap\Or$ is also an 
$\Or$-ideal prime to $r$.
\end{Prop}
\begin{proof}
Consider the natural injection 
\[
\phi \colon 
\Or/\afrak'\cap\Or 
\hookrightarrow 
\Or'/\afrak'.
\]
Let 
$\psi \colon \Or'/\afrak' \to \Or'/\afrak'$ be 
the multiplication map by $r$.
By the structure theorem for finite Abelian groups, 
we see that 
$\psi$ is an isomorphism.

Now we claim that 
$\psi^{-1}\circ\phi$ 
is surjective. 
Let $x\in\Or'$.
Since $r\Or' \subset \Or$, we see that $r x\in\Or$. 
Hence 
we have
\[
\psi^{-1}\circ\phi \;(rx+\afrak'\cap\Or)=\psi^{-1}(rx+\afrak')=x+\afrak'.
\]
Therefore the claim is proved. 

Since $\psi^{-1}\circ\phi$ is surjective and $\psi^{-1}$ is injective, 
we can see that $\phi$ is surjective. Thus $\phi$ is an isomorphism. 
Hence we get 
$N(\afrak'\cap\Or)=|\Or/\afrak'\cap\Or|=|\Or'/\afrak'|=N(\afrak')$.
\end{proof}

\begin{Lem}\label{sonaraissyu}
Let $\Or$, $\Or'$ be orders in an imaginary quadratic field with $\Or\subset\Or'$, 
and $r$ 
the relative conductor of $\Or$ in $\Or'$. 
Let $\afrak'$ be an 
$\Or'$-ideal 
prime to $r$.  
Then
\[
(\afrak'\cap\Or)\Or' = \afrak'.
\]
\end{Lem}
\begin{proof}
It follows that 
$\afrak'\cap\Or$ is prime to $r$ 
from Proposition \ref{sonaranorm}.
Thus $\afrak'\cap\Or+r\Or = \Or$. Hence we have
\begin{align*}
\afrak' = \afrak'\Or &= \afrak' (\afrak'\cap\Or + r\Or) \\
                             &= \afrak' (\afrak'\cap\Or) + r\afrak'\Or \\
                    &\subset \Or' (\afrak'\cap\Or) + r\afrak'\Or' \\
                             &= (\afrak'\cap\Or)\Or' + r\afrak' .
\end{align*}
Furthermore we see
\[
r\afrak' \subset r\Or' \subset \Or.
\]
Thus we get
\[
r\afrak' = 
\afrak'\cap r\afrak' \subset \afrak'\cap\Or \subset (\afrak'\cap\Or)\Or'.
\]
Hence we have
\begin{align*}
\afrak' \subset (\afrak'\cap\Or)\Or' + (\afrak'\cap\Or)\Or' = (\afrak'\cap\Or)\Or'.
\end{align*}
The other inclusion is obvious.
\end{proof}

\begin{Prop}\label{sonaraproper}
Let $\Or$, $\Or'$ be orders in an imaginary quadratic field with $\Or\subset\Or'$, 
and $r$ the relative conductor of $\Or$ in $\Or'$. 
If $\afrak'$ is a proper $\Or'$-ideal 
prime to $r$, 
then $\afrak'\cap\Or$ is also 
a proper $\Or$-ideal.
\end{Prop}
\begin{proof}
Since $\afrak'$ is proper, 
we see that $\Or'=\{\beta\in K \mid \beta\afrak'\subset\afrak'\}$. 
Proposition \ref{sonaranorm} implies that $\afrak'\cap\Or$ is prime to $r$. 
Thus $(\afrak'\cap\Or)+r\Or=\Or$. 
Let $\beta\in K$ satisfy $\beta(\afrak'\cap\Or)\subset \afrak'\cap\Or$. 
By Lemma \ref{sonaraissyu}, we obtain
\[
\beta\afrak'=\beta(\afrak'\cap\Or)\Or' \subset (\afrak'\cap\Or)\Or'=\afrak'.
\]
It follows that 
$\beta\in\Or'$.
We thus have
\begin{align*}
\beta\Or
&=\beta\bigl((\afrak'\cap\Or)+r\Or \bigr) \\
&=\beta(\afrak'\cap\Or) + \beta r\Or \\
&\subset (\afrak'\cap\Or) + r\Or'.
\end{align*}
However $r\Or'\subset\Or$, 
which proves that $\beta\Or \subset \Or$. 
Thus we see that $\beta\in\Or$. 
Therefore we see that 
$\Or \supset \{\beta\in K \mid \beta(\afrak'\cap\Or) \subset \afrak'\cap\Or\}$. 
The other inclusion is obvious. 
Thus $\afrak'\cap\Or$ is proper.
\end{proof}

\begin{Prop}\label{sonaraseki}
Let $\Or$, $\Or'$ be orders in an imaginary quadratic field with $\Or\subset\Or'$, 
and $r$ the relative conductor of $\Or$ in $\Or'$. 
Let $\afrak',\bfrak'$ be proper $\Or'$-ideals prime to $r$. Then
\[
\afrak'\bfrak'\cap\Or = (\afrak'\cap\Or) (\bfrak'\cap\Or).
\]
\end{Prop}
\begin{proof}
First, 
we have
\[
\afrak'\bfrak'\cap\Or \supset (\afrak'\cap\Or) (\bfrak'\cap\Or).
\]
Proposition \ref{sonaraproper} implies that 
$\afrak'\cap\Or$, $\bfrak'\cap\Or$ are proper $\Or$-ideals. 
By Proposition \ref{normseki}, we get
\[
N(\afrak'\bfrak') 
=N(\afrak') N(\bfrak'), \; 
N\bigl( (\afrak'\cap\Or)(\bfrak'\cap\Or) \bigr)
=N(\afrak'\cap\Or) N(\bfrak'\cap\Or) .
\]
Proposition \ref{sonaranorm} implies that
\[
N(\afrak'\cap\Or)=N(\afrak'), \;
N(\bfrak'\cap\Or)=N(\bfrak'), \;
N(\afrak'\bfrak'\cap\Or)=N(\afrak'\bfrak').
\]
We thus have
\begin{align*}
N(\afrak'\bfrak'\cap\Or)
&=N(\afrak'\bfrak') 
=N(\afrak') N(\bfrak') 
=N(\afrak'\cap\Or) N(\bfrak'\cap\Or) 
=N\bigl( (\afrak'\cap\Or)(\bfrak'\cap\Or) \bigr).
\end{align*}
Therefore, we see that 
\[
\afrak'\bfrak'\cap\Or = (\afrak'\cap\Or) (\bfrak'\cap\Or).
\]
This completes the proof.
\end{proof}

\section{Decomposition of proper $\Or$-ideals}\label{Decomp}
We prove Theorem \ref{yosoCor} in this section. To prove this theorem, we prepare the following lemma.

\begin{Lem}\label{result2}
Let $\Or$ be the order of conductor $f=l^\lambda$ in an imaginary quadratic field $K$, 
where 
$\lambda >0$ and the $l$ is a prime, 
and let $D$ be the discriminant of $\Or$.
Let $\afrak$ be a proper $\Or$-ideal, 
and 
let $m$ be an arbitrary positive integer. 
Write 
\[
m=n l^{h}, 
\]
where 
\begin{itemize}
\item $n$ is an integer relatively prime to $f$, 
\item $h \geq0$. 
\end{itemize}
Then, the followings are equivalent.
\begin{enumerate}[label=$(\mathrm{\roman*})$]
\item $N(\afrak)=m$.
\item There exist 
\begin{itemize}
\item a proper $\Or$-ideal $\bfrak$ where $N(\bfrak)=n$, 
\item a proper $\Or$-ideal $\cfrak$ where $N(\cfrak)=l^h$, 
\end{itemize}
such that 
\[
\afrak=\bfrak\cfrak.
\]
\end{enumerate}
\end{Lem}

\begin{proof}
$[(\mathrm{i})\Rightarrow(\mathrm{ii})]$

\noindent
Proposition \ref{norm} implies that 
$\afrak\Or_K$ is an $\Or_K$-ideal 
satisfying 
$N(\afrak\Or_K)=m$.
By Theorem \ref{result0}, 
there exist proper $\Or_K$-ideals $\bfrak'$, $\cfrak'$, 
where $N(\bfrak')=n$ and $N(\cfrak')={l}^{h}$ 
such that 
\[
\afrak\Or_K=\bfrak' \cfrak'.
\]
Suppose that $\cfrak
=\afrak\cdot
 (\bfrak'\cap\Or)^{-1}$. 
Note that Proposition \ref{sonaraproper} implies that $\bfrak'\cap\Or$ is proper. 
Since $\afrak$ is proper, we see that $\cfrak$ is proper. 

Now we claim that $\cfrak$ is integral. 
We need to show that $\afrak\subset\bfrak'\cap\Or$.
Since $\afrak\Or_K=\bfrak' \cfrak'$, 
we see that $\afrak\Or_K \subset \bfrak'$. 
Thus we have
\[
\afrak \subset \afrak\Or_K\cap\Or \subset \bfrak'\cap\Or.
\]
The claim is proved.

Suppose that $\bfrak=\bfrak'\cap\Or$, then we have $\afrak=\bfrak\cfrak$. 
It remains to show that $N(\bfrak)=n$, $N(\cfrak)=l^{h}$.
Proposition \ref{sonaranorm} implies that 
\[
N(\bfrak)=N(\bfrak'\cap\Or)=N(\bfrak')=n.
\]
By Proposition \ref{normseki}, we get
\[
nl^{h}=m=N(\afrak)=N(\bfrak)N(\cfrak)=n N(\cfrak).
\]
Hence we see $N(\cfrak)=l^{h}$. 
This completes the proof.

\noindent
$[(\mathrm{ii})\Rightarrow(\mathrm{i})]$ 

\noindent
This follows immediately from Proposition \ref{normseki}.
\end{proof}

Now we prove the following key proposition. 
\begin{Prop}\label{yoso}
Let $\Or$ be the order of conductor 
$f$ 
in an imaginary quadratic field $K$. 
Write 
$f={l_1}^{\lambda_1}\cdots{l_t}^{\lambda_t}$, 
where 
$t\geq 0$, $\lambda_k >0$, and 
the $l_k$'s are distinct primes. 
Let  
$D$ be the discriminant of $\Or$. 
Let $\afrak$ be a proper $\Or$-ideal, 
and 
let $m$ be an arbitrary positive integer. 
Write 
\[
m=n \cdot {l_1}^{h_1}\cdots{l_t}^{h_t}, 
\]
where 
\begin{itemize}
\item the $n$ is an integer relatively prime to $f$, 
\item 
$h_1$, $\ldots$ , $h_t \geq0$.
\end{itemize}
Then, the followings are equivalent.
\begin{enumerate}[label=$(\mathrm{\roman*})$]
\item $N(\afrak)=m$.
\item There exist 
\begin{itemize}
\item a proper $\Or$-ideal $\bfrak$ where $N(\bfrak)=n$, 
\item proper $\Or$-ideals $\cfrak_1$, $\ldots$ , $\cfrak_t$, 
where $N(\cfrak_1)={l_1}^{h_1}$, $\ldots$ , $N(\cfrak_{t})={l_{t}}^{h_{t}}$, 
\end{itemize}
such that 
\[
\afrak=\bfrak \cdot \cfrak_1 \cdots \cfrak_t.
\]
\end{enumerate}
\end{Prop}

\begin{proof}
$[(\mathrm{i})\Rightarrow(\mathrm{ii})]$

\noindent
Let $\Or_k$ be the order of 
discriminant $({l_1}^{\lambda_1}\cdots{l_k}^{\lambda_k})^2 d_K$. 
Then we have 
\[
\Or=\Or_t \subset \Or_{t-1} \subset \dots \subset \Or_1 \subset \Or_K.
\] 
Propositions \ref{proper} and \ref{norm} imply that
$\afrak\Or_{t-1}$ is a proper $\Or_{t-1}$-ideal 
satisfying 
$N(\afrak\Or_{t-1})=m$, when $t\geq 2$. 

We prove by induction on $t$.
Lemma \ref{result2} implies that the case $t=1$ 
holds. 
Now we prove the case $t\geq2$.
By the 
assumption of induction, 
there exist proper $\Or_{t-1}$-ideals $\bfrak'$, $\cfrak'_1$, $\ldots$ , $\cfrak'_{t-1}$, 
where $N(\bfrak')=n {l_t}^{h_t}$, 
$N(\cfrak'_1)={l_1}^{h_1}$, $\ldots$ , $N(\cfrak'_{t-1})={l_{t-1}}^{h_{t-1}}$ 
such that 
\[
\afrak\Or_{t-1}=\bfrak' \cdot \cfrak'_1 \cdots \cfrak'_{t-1}.
\]
Note that the conductor of $\Or_{t-1}$ is 
${l_1}^{\lambda_1}\!\!\cdots{l_{t-1}}^{\lambda_{t-1}}$.
By Corollary \ref{result1}, 
there exist proper $\Or_{t-1}$-ideals $\bfrak''$, $\cfrak'_t$, 
where $N(\bfrak'')=n$ and $N(\cfrak'_t)={l_t}^{h_t}$ 
such that 
\[
\bfrak'=\bfrak'' \cfrak'_t.
\] 
Thus we can write
\[
\afrak\Or_{t-1}=\bfrak''\cdot \cfrak'_1 \cdots \cfrak'_{t-1}\cfrak'_t.
\]
Suppose that 
\[
\cfrak_t
=\afrak\cdot
 (\bfrak''\cap\Or_t)^{-1} (\cfrak'_1\cap\Or_t)^{-1} \cdots (\cfrak'_{t-1}\cap\Or_t)^{-1}.
\] 
Note that 
$\bfrak''\cap\Or_t$, $\cfrak'_1\cap\Or_t$, $\dots$ ,$\cfrak'_{t-1}\cap\Or_t$ 
are proper 
by Proposition \ref{sonaraproper}.
Since $\afrak$ is proper, we see that $\cfrak_t$ is proper. 

Now we claim that $\cfrak_t$ is integral. 
We need to show that 
$\afrak \subset (\bfrak''\cap\Or_t)(\cfrak'_1\cap\Or_t)\cdots(\cfrak'_{t-1}\cap\Or_t)$.
Since $\afrak\Or_{t-1}=\bfrak''\cdot \cfrak'_1 \cdots \cfrak'_{t-1}\cfrak'_t$, 
we see that $\afrak\Or_{t-1} \subset \bfrak''\cdot \cfrak'_1 \cdots \cfrak'_{t-1}$. 
Thus we have
\begin{align*}
\afrak &\subset (\afrak\Or_{t-1})\cap\Or_t \\
         &\subset (\bfrak''\cdot \cfrak'_1 \cdots \cfrak'_{t-1})\cap\Or_t \\
         &= (\bfrak''\cap\Or_t)(\cfrak'_1\cap\Or_t)\cdots(\cfrak'_{t-1}\cap\Or_t).
\end{align*}
The third line follows from Proposition \ref{sonaraseki}. 
The claim is proved. 

Suppose that $\bfrak=\bfrak''\cap\Or_t$, 
$\cfrak_1=\cfrak'_1\cap\Or_t$, $\ldots$ , $\cfrak_{t-1}=\cfrak'_{t-1}\cap\Or_t$, 
then we have $\afrak=\bfrak \cdot \cfrak_1 \cdots \cfrak_t$. 
It remains to show that 
$N(\bfrak)=n$, $N(\cfrak_1)={l_1}^{h_1}$, $\ldots$ , $N(\cfrak_{t})={l_{t}}^{h_{t}}$. 
Proposition \ref{sonaranorm} implies that 
\begin{align*}
N(\bfrak)=N(\bfrak''\cap\Or_t)&=N(\bfrak'')=n, \\
N(\cfrak_1)=N(\cfrak'_1\cap\Or_t)&=N((\cfrak'_1)={l_1}^{h_1}, \\
&\vdots \\
N(\cfrak_{t-1})=N(\cfrak'_{t-1}\cap\Or_t)&=N(\cfrak'_{t-1})={l_{t-1}}^{h_{t-1}}.
\end{align*}
By Proposition \ref{normseki}, we get
\[
n \cdot {l_1}^{h_1}\cdots{l_{t-1}}^{h_{t-1}} {l_t}^{h_t}=m=N(\afrak)
=N(\bfrak)\cdot N(\cfrak_1)\cdots N(\cfrak_{t-1})N(\cfrak_{t})
=n\cdot {l_1}^{h_1}\cdots{l_{t-1}}^{h_{t-1}} N(\cfrak_{t}).
\]
Hence we see $N(\cfrak_t)={l_t}^{h_t}$. 
This completes the proof.

\noindent
$[(\mathrm{ii})\Rightarrow(\mathrm{i})]$ 

\noindent
This follows immediately from Proposition \ref{normseki}.
\end{proof}

Theorem \ref{yosoCor} follows immediately from 
Corollary \ref{result1} and Proposition \ref{yoso}.

\section{Proof of the main theorem}\label{ProofoftheM}%
In this section, we prove Theorem \ref{MainTheorem}-[(i)$\Rightarrow$(ii)]. 
To prove this, 
we prepare the following lemma. 

\begin{Lem}
Let $\Or$ be the order of conductor 
$f$ 
in an imaginary quadratic field $K$, 
Write 
$f={l_1}^{\lambda_1}\cdots{l_t}^{\lambda_t}$, 
where 
$t\geq 0$, $\lambda_k >0$, and 
the $l_k$'s are distinct primes. 
Let 
$D$ be the discriminant of $\Or$. 
Let $\afrak$ be a proper $\Or$-ideal, 
and 
let $m$ be an arbitrary positive integer. 
Write 
\[
m=p_1\cdots p_r \cdot {q_1}^{e_1}\cdots{q_s}^{e_s} \cdot {l_1}^{h_1}\cdots{l_t}^{h_t}
\]
where
\begin{itemize}
\item the $p_i$'s are primes relatively prime to $f$ with $(D/p_i)=0,1$, 
\item the $q_j$'s are distinct primes relatively prime to $f$ with $(D/q_j)=-1$, 
\item $r, s\geq 0$, $e_j >0$, $h_k \geq0$. 
\end{itemize}
Then $(\mathrm{i})$ implies $(\mathrm{ii})$. 
\begin{enumerate}[label=$(\mathrm{\roman*})$]
\item $N(\afrak)=m$.
\item All $e_j$'s are even, and 
there exist 
\begin{itemize}
\item prime $\Or$-ideals $\pfrak_i$'s where $N(\pfrak_i)=p_i$, 
\item proper $\Or$-ideals $\cfrak_k$'s where $N(\cfrak_k)={l_k}^{h_k}$, 
\end{itemize}
such that 
\[
\text{$[\afrak]=[\pfrak_1]\cdots[\pfrak_r] \cdot [\cfrak_1] \cdots [\cfrak_t]$ in $C(\Or)$}.
\]
\end{enumerate}
\end{Lem}

\begin{proof}
Theorem \ref{yosoCor} implies that 
all $e_j$'s are even, and 
there exist prime $\Or$-ideals $\pfrak_i$'s, 
proper $\Or$-ideals $\cfrak_k$'s, 
where $N(\pfrak_i)=p_i$ and $N(\cfrak_k)={l_k}^{h_k}$ 
such that 
$\afrak=\pfrak_1\cdots\pfrak_r \cdot (q_1\Or)^{e_1/2}\cdots (q_s\Or)^{e_s/2} 
\cdot \cfrak_1 \cdots \cfrak_t$.
Thus we have
\begin{align*}
[\afrak]
&=[\pfrak_1]\cdots[\pfrak_r] \cdot [q_1\Or]^{e_1/2}\cdots [q_s\Or]^{e_s/2} 
\cdot [\cfrak_1] \cdots [\cfrak_t] \\
&=[\pfrak_1]\cdots[\pfrak_r] \cdot 1^{e_1/2}\cdots 1^{e_s/2}
\cdot [\cfrak_1] \cdots [\cfrak_t] \\
&=[\pfrak_1]\cdots[\pfrak_r]\cdot [\cfrak_1] \cdots [\cfrak_t] \;\; \text{in $C(\Or)$}. 
\end{align*}
This competes the proof.
\end{proof}

Theorem \ref{MainTheorem}-[(i)$\Rightarrow$(ii)] follows immediately by 
this lemma and Propositions \ref{existO} and \ref{C(D),C(O)}.

\section{Examples of the main theorem}\label{ExamplesoftheM}
Now we prove Example \ref{exa-4}. 

\begin{proof}[Proof of Example \rm{\ref{exa-4}}]
We see that the discriminant of $x^2+y^2$ is $D=-4$. 
Since $K=\Q(\sqrt{D})=\Q(\sqrt{-1})$, 
we have $f=\sqrt{D/d_K}=1$. 
Note that $C(-4)=\{[x^2+y^2]\}\simeq\{1\}$ (see Cox \cite[\S2-A]{Cox}), 
and for a prime $p$, 
the following holds.
\[
\text{$p=x^2+y^2$ has an integer solution 
$\iff$ 
$p=2$ or $p\equiv1 \bmod4$.}
\]
Since the group $C(-4)$ is trivial, 
we see that for positive definite forms $f_i(x,y)$'s of discriminant $D=-4$, 
\[
\text{$[x^2+y^2]=1\cdots1=[x^2+y^2]\cdots[x^2+y^2]=[f_1(x,y)]\cdots[f_r(x,y)]$ in $C(-4)$}.
\]
Then the assertion follows from Theorem \ref{MainTheorem}. 
\end{proof}

Next we see the following example, 
where the form class group $C(D)\neq\{1\}$ 
but the conductor $f=1$. 

\begin{Exa}\label{exa-20}
Let $m$ be an arbitrary positive integer. 
Write 
\[
m=p_1\cdots p_r \cdot {q_1}^{e_1}\cdots{q_s}^{e_s},
\]
where
\begin{itemize}
\item the $p_i$'s are primes with $p_i=2, 5$ or $p_i\equiv1, 3, 7, 9 \bmod20$, 
\item the $q_j$'s are distinct primes with $q_j\equiv11, 13, 17, 19 \bmod20$, 
\item $r, s\geq 0$, $e_j >0$.
\end{itemize}
Then, the followings are equivalent.
\begin{enumerate}[label=$(\mathrm{\roman*})$]
\item $m=2x^2+2xy+3y^2$ has an integer solution.
\item All $e_j$'s are even, and 
the number of primes $p_k$'s with 
$p_k=2$ or $p_k\equiv3, 7 \bmod20$ is odd.
\end{enumerate}
\end{Exa}
\begin{proof}
We see that the discriminant of $2x^2+2xy+3y^2$ is $D=-20$. 
Since $K=\Q(\sqrt{D})=\Q(\sqrt{-5})$, 
we have $f=\sqrt{D/d_K}=1$. 
Note that $C(-20)=\{[x^2+5y^2], [2x^2+2xy+3y^2]\}\simeq\{1, -1\}$ (see Cox \cite[\S2-A]{Cox}), 
and 
for a prime $p$, 
the followings hold.
\begin{align*}
\text{$p=x^2+5y^2$ has an integer solution} 
&\iff 
\text{$p=5$ or $p\equiv1, 9\bmod20$}. \\
\text{$p=2x^2+2xy+3y^2$ has an integer solution}
&\iff 
\text{$p=2$ or $p\equiv3, 7 \bmod20$}.
\end{align*}
Thus we see that for positive definite forms $f_i(x,y)$'s of discriminant $D=-20$, the following holds.
\begin{align*}
&\text{$[2x^2+2xy+3y^2]=[f_1(x,y)]\cdots[f_r(x,y)]$} \\
\iff 
&\text{The number of forms $f_k(x,y)$'s with $[f_k(x,y)]=[2x^2+2xy+3y^2]=-1$ is odd} \\
\iff 
&\text{The number of primes $p_k(x,y)$'s with $p_k=2$ or $p_k\equiv3, 7 \bmod20$ is odd}.
\end{align*}
Then the assertion follows from Theorem \ref{MainTheorem}. 
\end{proof}

Strictly speaking, since $f=1$, we can derive Example \ref{exa-20} from Theorem \ref{result0}, Propositions \ref{existO} and \ref{C(D),C(O)}. 

Now we see examples 
where 
$f>1$. 
Let $D$, $D'\equiv 0,1 \bmod4$ be negative integers 
satisfying 
$\sqrt{D/D'}\in\Z$. 
Suppose that $r=\sqrt{D/D'}$. 
By Propositions \ref{existO} and \ref{C(D),C(O)}, 
there exist orders $\Or$, $\Or'$ such that 
$\Or \subset\Or'$ and
\[
C(D) \simeq C(\Or), \;\; C(D') \simeq C(\Or').
\] 

\begin{Prop}\label{surj}
Let $C(\Or)$, $C(\Or')$ be as above. 
Consider the homomorphism 
\[
C(\Or) \rightarrow 
C(\Or'), \; [\afrak]\mapsto[\afrak\Or'].
\]
Then the homomorphism is surjective. 
\end{Prop}
\begin{proof}
To prove the proposition, we prepare the following lemma: 
\begin{Lem}
Let $\Or$ be an order in an imaginary quadratic field. Given a nonzero integer $M$, then every ideal class in $C(\Or)$ contains a proper $\Or$-ideal whose norm is relatively prime to $M$.
\end{Lem}
\begin{proof}
See Cox \cite[Corollary 7.17]{Cox}.
\end{proof}
Proposition \ref{surj} follows immediately from the lemma and Proposition \ref{sonaraissyu}.
\end{proof}

Thus we can consider the surjection
\[
\pi \colon C(D) \simeq C(\Or) \twoheadrightarrow C(\Or') \simeq C(D').
\]
The above surjection $\pi$ is ``consistent'' with integers represented by forms; if a form class in $C(D)$ represents an integer, then its image in $C(D')$ also represents the same integer (see Proposition \ref{arawasu}), 
although the opposite side  is a naive problem (see Propositions \ref{zuretaraarawasu} and \ref{zuretaraarawasanai}).

\begin{Prop}\label{arawasu}
Let $C(D)$, $C(D')$ be as above. 
Assume that 
$\pi \colon[f(x,y)] \mapsto [g(x,y)]$. 
If $f(x,y)$ represents an integer $m$, 
then $g(x,y)$ represents the integer $m$.
\end{Prop}
\begin{proof}
This follows immediately from 
Propositions \ref{existO}, \ref{C(D),C(O)}, and \ref{norm}.
\end{proof}

\begin{Lem}\label{sento}
A form $f(x,y)$ properly represents an integer m if and only if $f(x,y)$ is properly equivalent to the form $mx^2+bxy+cy^2$ for some b,c $\in \Z$.
\end{Lem}
\begin{proof}
See Cox \cite[Lem.2.3]{Cox}.
\end{proof}

\begin{Lem}\label{hensuhenkan}
Let $C(D)$, $C(D')$ be as above. 
Assume that 
$\pi \colon[f(x,y)] \mapsto [g(x,y)]$. 
Then there is an element 
$
\begin{pmatrix}
 s & t \\
 u & v 
\end{pmatrix}
$
$\in \rm{GL}(2,\Z)$ such that
\begin{equation}
f(x,y)=g(sx+ty,ux+vy), \quad
\mathrm{det}
\begin{pmatrix}
 s & t \\
 u & v 
\end{pmatrix}
=\pm r. \label{1}
\end{equation}
\end{Lem}

\begin{proof}
By Chebotarev Density Theorem, $f(x,y)$ represents 
a prime $p$ with $p\nmid D'$, 
or a prime square $q^2$ with $q\nmid D'$. 

When $f(x,y)$ represents a prime $p$, 
we claim that $(\ref{1})$ holds. 
By Lemma \ref{sento}, we can assume that 
\begin{equation}
\text{$[f(x,y)]=[px^2+bxy+cy^2]$, where $b,c\in\Z$}.\label{2}
\end{equation} 
Proposition \ref{arawasu} implies that $[g(x,y)]$ represents $p$. 
By Lemma \ref{sento}, we can assume that 
\begin{equation}
\text{$[g(x,y)]=[px^2+Bxy+Cy^2]$, where $B,C\in\Z$}. \label{3}
\end{equation}
By Proposition \ref{C(D),C(O)}, we have
\[
C(D) \ni [px^2+bxy+cy^2] \mapsto [\langle p,(-b+\sqrt{D})/2 \rangle_\Z] \in C(\Or), 
\]
\[
C(D') \ni [px^2+Bxy+Cy^2] \mapsto [\langle p,(-B+\sqrt{D'})/2 \rangle_\Z] \in C(\Or').
\]
By Proposition \ref{surj}, we get
\[
C(\Or) \ni [\langle p,(-b+\sqrt{D})/2 \rangle_\Z] 
\mapsto    [\langle p,(-b+\sqrt{D})/2 \rangle_\Z \Or'] 
                =[\langle p,(-B+\sqrt{D'})/2 \rangle_\Z] \in C(\Or').
\]
Since there are only two $\Or'$-ideals whose norm is $p$, 
\[
\langle p,(-b+\sqrt{D})/2 \rangle_\Z \Or' 
= \langle p,(-B \pm\sqrt{D'})/2 \rangle_\Z.
\]
Hence there exist $s$, $t\in\Z$ such that 
\[
(-b+\sqrt{D})/2 = s p + t (-B\pm\sqrt{D'})/2.
\]
Since $D=r^2 D'$, we get
\[
-b+r\sqrt{D'} = 2s p - t B \pm t\sqrt{D'}.
\]
Hence we have 
$b=-2sp+tB$ and $\pm r=t$. 
Now we can write
\begin{align*}
&p(x-sy)^2+B(x-sy)(ty)+C(ty)^2 \\
=&
\left(\begin{array}{cc} x-sy&ty\\\end{array} \right)
\left(\begin{array}{cc} p&B/2\\ B/2&C\\\end{array} \right)
\left(\begin{array}{cc} x-sy\\ ty\\\end{array} \right) \\
=&
\left(\begin{array}{cc} x&y\\\end{array} \right)
\left(\begin{array}{cc} 1&0\\ -s&t\\\end{array} \right)
\left(\begin{array}{cc} p&B/2\\ B/2&C\\\end{array} \right)
\left(\begin{array}{cc} 1&-s\\ 0&t\\\end{array} \right)
\left(\begin{array}{cc} x\\ y\\\end{array} \right) \\
=&
\left(\begin{array}{cc} x&y\\\end{array} \right)
\left(\begin{array}{cc} p&-sp+t(B/2)\\ -sp+t(B/2)&ps^2-Bst+Ct^2\\\end{array} \right)
\left(\begin{array}{cc} x\\ y\\\end{array} \right) \\
=&
px^2+(-2sp+tB)xy+(ps^2-Bst+Ct^2)y^2 \\
=&
px^2+bxy+(ps^2-Bst+Ct^2)y^2.
\end{align*}
The discriminant of $px^2+bxy+(ps^2-Bst+Ct^2)y^2$ 
is 
\begin{align*}
-4\cdot
\mathrm{det}\Bigl\{
\left(\begin{array}{cc} 1&0\\ -s&t\\\end{array} \right)
\left(\begin{array}{cc} p&B/2\\ B/2&C\\\end{array} \right)
\left(\begin{array}{cc} 1&-s\\ 0&t\\\end{array} \right) \Bigr\}
&= t^2(B^2-4pC) \\
&= r^2 D' \\
&=D.
\end{align*}
Since the discriminant of $px^2+bxy+cy^2$ is also $D$, we see that 
\[
px^2+bxy+(ps^2-Bst+Ct^2)y^2 
= px^2+bxy+cy^2.\\
\]
Therefore we see that 
\begin{align}
&p(x-sy)^2+B(x-sy)(ty)+C(ty)^2 \notag\\
=&
px^2+bxy+cy^2.
\label{4}
\end{align}
Furthermore we have
\begin{equation}
\mathrm{det}
\begin{pmatrix}
 1 & -s \\
 0 & t 
\end{pmatrix}
=t
=\pm r.\label{5}
\end{equation}
Thus 
the claim is proved by $(\ref{2}), (\ref{3}), (\ref{4})$ and $(\ref{5})$. 

When $f(x,y)$ represents a prime square $q^2$, 
we can prove similarly. 
\end{proof}

\begin{Prop}\label{zuretaraarawasu}
Let $C(D)$, $C(D')$ be as above. 
Assume that $r=\sqrt{D/D'}=l^{\lambda}$, 
where $\lambda>0$ and 
$l$ is a prime, 
and assume that 
$\pi \colon [f(x,y)] \mapsto [g(x,y)]$. 
If $g(x,y)$ represents $l^h$ where $h\geq0$, 
then $f(x,y)$ represents $l^{2\lambda+h}$.
\end{Prop}

\begin{proof}
By Lemma \ref{hensuhenkan}, 
there is an element 
$
\begin{pmatrix}
 s & t \\
 u & v 
\end{pmatrix}
$
$\in \rm{GL}(2,\Z)$ such that
\begin{equation*}
f(x,y)=g(sx+ty,ux+vy), \quad
\mathrm{det}
\begin{pmatrix}
 s & t \\
 u & v 
\end{pmatrix}
=\pm l^\lambda.
\end{equation*}
Thus we see that 
\begin{equation}
f(vx-ty,-ux+sy)=g(\pm l^\lambda x,\pm l^\lambda y). \label{6}
\end{equation}
Assume that 
$g(x,y)$ represents $l^h$. 
Namely 
there exist $X$, $Y\in\Z$ such that $g(X,Y)=l^h$. 
Hence we have 
$
g(\pm l^\lambda X,\pm l^\lambda Y)
=l^{2\lambda} \cdot g(X,Y)=l^{2\lambda+h}$. 
Thus (\ref{6}) implies that 
\[
f(vX-tY,-uX+sY)=l^{2\lambda+h}.
\]
Therefore we see that $f(x,y)$ represents $l^{2\lambda+h}$.
\end{proof}


\begin{Prop}\label{zuretaraarawasanai}
Let $C(D)$, $C(D')$ be as above. 
Assume that $r=\sqrt{D/D'}=l^{\lambda}$, 
where $\lambda>0$ and 
$l$ is a prime. 
Let $[f(x,y)]\in C(D)$ be a form class, then 
$f(x,y)$ does not represent $l^{h}$ 
for any odd integer $1\leq h<2\lambda+1$. 
\end{Prop}
\begin{proof}
Let $h$ be an odd integer with $1\leq h<2\lambda+1$, 
and 
assume that $f(x,y)$ represents $l^h$. 

If $l^h$ is not properly represented by $f(x,y)$, 
there exists 
another 
odd integer $h'$ with 
$1\leq h' < h$
such that 
$l^{h'}$ is properly represented by $f(x,y)$
. 
Thus we may assume that $l^h$ is properly represented by $f(x,y)$. 
By Lemma \ref{sento}, 
we can assume that $[f(x,y)]=[l^h x^2+bxy+cy^2]$ 
where $b, c\in\Z$. 
Thus 
\[
b^2-4 l^h c=D=(l^{\lambda})^2 D'.
\]
Now we suppose that 
$h=2j+1$ where $0\leq j <\lambda$. Then we have
\[
b^2-4 l^{2j+1} c = l^{2\lambda} D'.
\]
Since $j<\lambda$, we see $2j+1 < 2\lambda$. 
It follows that $b$ can be written as
\[
\text{$b=l^{j+1}b'$, where $b'\in\Z$.}
\] 
Thus we get
\[
l^{2j+2}{b'}^2-4 l^{2j+1} c = l^{2\lambda} D'.
\]

When $l$ is odd, 
we see $l|c$. 
Hence $l^h x^2+bxy+cy^2$ is not primitive. 
However, any form properly equivalent to a primitive form is itself primitive, 
this is a contradiction. 

When $l=2$, we see 
\[
2^{2j+2}b'^2-2^{2j+3}c=l^{2\lambda}D'.
\]
Thus we have
\[
b'^2-2c=2^{2(\lambda-j-1)}D'.
\]
Now we claim that $c$ is even. If $c$ is odd, we have 
$c\equiv\pm1 \bmod4$. Hence
\[
b'^2\pm2 \equiv 2^{2(\lambda-j-1)}D' \bmod 4.
\]
Since $b'^2\equiv0, 1 \bmod4$, $2^{2(\lambda-j-1)}\equiv0, 1 \bmod4$ 
and $D'\equiv0, 1 \bmod4$, 
we see a contradiction. 
Thus the claim is proved. 
Hence $l^h x^2+bxy+cy^2$ is not primitive. 

Therefore the proposition is proved.
\end{proof}

Now we see the following 
examples of the main theorem in the case $f>1$. 
First we prove Example \ref{exa-32}. 

\begin{proof}[Proof of Example \rm{\ref{exa-32}}]
We see that the discriminant of $3x^2+2xy+3y^2$ is $D=-32$. 
Since $K=\Q(\sqrt{D})=\Q(\sqrt{-2})$, 
we have $f=\sqrt{D/d_K}=2$. 
Note that $C(-32)=\{[x^2+8y^2], [3x^2+2xy+3y^2]\}\simeq\{1, -1\}$, and 
for a prime $p$, 
the followings hold.
\begin{align*}
\text{$p=x^2+8y^2$ has an integer solution} 
&\iff 
p\equiv1 \bmod8. \\
\text{$p=3x^2+2xy+3y^2$ has an integer solution} 
&\iff 
p\equiv3 \bmod8. 
\end{align*}
Now consider the surjection
\begin{align*}
&C(-32)\twoheadrightarrow C(-8), \;\; \\
&\text{$[x^2+8y^2]\mapsto[x^2+2y^2]$}, \\
&\text{$[3x^2+2xy+3y^2]\mapsto[x^2+2y^2]$.}
\end{align*}
We easily see that 
$x^2+2y^2$ represents $2^0$, $2^1$. 
By Proposition \ref{zuretaraarawasu}, 
it follows that both $x^2+8y^2$ and $3x^2+2xy+3y^2$ represent $2^2$, $2^3$. 
By Proposition \ref{zuretaraarawasanai}, 
it follows that both $x^2+8y^2$ and $3x^2+2xy+3y^2$ do not represent $2^1$. 
Then the assertion follows from Theorem \ref{MainTheorem}. 
\end{proof}

Next we see the following example. 

\begin{Exa}
Let $m$ be an arbitrary positive integer. 
Write
\[
m=p_1\cdots p_r \cdot {q_1}^{e_1}\cdots{q_s}^{e_s} \cdot {2}^{h}
\]
where
\begin{itemize}
\item the $p_i$'s are primes with $p_i\equiv1, 5 \bmod8$, 
\item the $q_j$'s are distinct primes with $q_j\equiv3, 7 \bmod8$, 
\item $r, s\geq 0$, $e_j >0$, $h \geq0$. 
\end{itemize}
Then, the followings are equivalent.
\begin{enumerate}[label=$(\mathrm{\roman*})$]
\item $m=4x^2+4xy+5y^2$ has an integer solution.
\item All $e_j$'s are even, and 
one of the following holds.
\begin{enumerate}[label=$(\mathrm{\alph*})$]
\item the number of  primes $p_k$'s with $p_k\equiv5 \bmod8$ is odd, 
$h=0$, 
\item $h=2$,
\item $h\geq4$.
\end{enumerate}
\end{enumerate}
\end{Exa}

\begin{proof}
We see that the discriminant of $4x^2+4xy+5y^2$ is $D=-64$. 
Since $K=\Q(\sqrt{D})=\Q(\sqrt{-1})$, 
we have $f=\sqrt{D/d_K}=2^2$. 
Note that $C(-64)=\{[x^2+16y^2], [4x^2+4xy+5y^2]\}\simeq\{1, -1\}$, and 
for a prime $p$, 
the followings hold.
\begin{align*}
\text{$p=x^2+16y^2$ has an integer solution} 
&\iff 
p\equiv1 \bmod8. \\
\text{$p=4x^2+4xy+5y^2$ has an integer solution} 
&\iff 
p\equiv5 \bmod8. 
\end{align*}
Now consider the surjection
\begin{align*}
&C(-64)\twoheadrightarrow C(-4), \;\; \\
&\text{$[x^2+16y^2]\mapsto[x^2+y^2]$}, \\
&\text{$[4x^2+4xy+5y^2]\mapsto[x^2+y^2]$.}
\end{align*}
We easily see that 
$x^2+y^2$ represents $2^0$, $2^1$. 
By Proposition \ref{zuretaraarawasu}, 
it follows that both $x^2+16y^2$ and $4x^2+4xy+5y^2$ represent $2^4$, $2^5$. 
By Proposition \ref{zuretaraarawasanai}, 
it follows that both $x^2+16y^2$ and $4x^2+4xy+5y^2$ do not represent $2^1$, $2^3$. 

Furthermore consider the surjection
\begin{align*}
&C(-64)\twoheadrightarrow C(-16), \;\; \\
&[x^2+16y^2]\mapsto[x^2+4y^2], \\
&[4x^2+4xy+5y^2]\mapsto[x^2+4y^2].
\end{align*}
We easily see that 
$x^2+4y^2$ represents $2^0$. 
By Proposition \ref{zuretaraarawasu}, 
it follows that both $x^2+16y^2$ and $4x^2+4xy+5y^2$ represent $2^2$. 
Then the assertion follows from Theorem \ref{MainTheorem}. 
\end{proof}


Finally we prove Example \ref{exa-108}, 
which is more complicated since the conductor (=6) is a composite number. 

\begin{proof}[Proof of Example \rm{\ref{exa-108}}]
We see that the discriminant of $4x^2+2xy+7y^2$ is $D=-108$. 
Since $K=\Q(\sqrt{D})=\Q(\sqrt{-3})$, 
we have $f=\sqrt{D/d_K}=6$. 
Note that $C(-108)=\{[x^2+27y^2], [4x^2+2xy+7y^2], [4x^2-2xy+7y^2]\}
\simeq\Z/3\Z$. 
Using cubic reciprocity, we see that 
for a prime $p$, 
the followings hold 
(see Ireland and Rosen \cite[Proposition 9.6.2]{Ire}). 
\begin{align*}
&\text{$p=x^2+27y^2$ has an integer solution} \\
\iff 
&\text{$p\equiv1 \bmod3$ and $2$ is a cubic residue modulo $p$.} \\
\hspace{5pt}\\
&\text{$p=4x^2 \pm2xy+7y^2$ has an integer solution} \\
\iff 
&\text{$p\equiv1 \bmod3$ and $2$ is not a cubic residue modulo $p$.} 
\end{align*}
Now consider the surjections
\begin{align*}
&C(-108)\twoheadrightarrow C(-27), \;\;          &&C(-27)\twoheadrightarrow C(-3),\\
&\text{$[x^2+27y^2]\mapsto[x^2+xy+7y^2]$},     &&[x^2+xy+7y^2]\mapsto[x^2+xy+y^2].\\
&\text{$[4x^2+2xy+7y^2]\mapsto[x^2+xy+7y^2]$}, \\
&\text{$[4x^2-2xy+7y^2]\mapsto[x^2+xy+7y^2]$.}
\end{align*}
We easily see that 
$x^2+xy+7y^2$ represents $2^0$. 
By Proposition \ref{zuretaraarawasu}, 
it follows that both $x^2+27y^2$ and $4x^2\pm2xy+7y^2$ represent $2^2$. 
Since $x^2+xy+y^2$ does not represent $2$, we see that $x^2+xy+y^2$ does not represent $2^h$ for any odd integer $h\geq1$. 
The contrapositive of Proposition \ref{arawasu} implies 
that both $x^2+27y^2$ and $4x^2\pm2xy+7y^2$ 
do not represent $2^h$ for any odd integer $h\geq1$. 

Furthermore consider the surjection
\begin{align*}
&C(-108)\twoheadrightarrow C(-12), \;\; \\
&\text{$[x^2+27y^2]\mapsto[x^2+3y^2]$}, \\
&\text{$[4x^2+2xy+7y^2]\mapsto[x^2+3y^2]$}, \\ 
&\text{$[4x^2-2xy+7y^2]\mapsto[x^2+3y^2]$.}
\end{align*}
We easily see that 
$x^2+3y^2$ represents $3^0$, $3^1$. 
By Proposition \ref{zuretaraarawasu}, 
it follows that both $x^2+27y^2$ and $4x^2\pm2xy+7y^2$ represent $3^2$, $3^3$. 
By Proposition \ref{zuretaraarawasanai}, 
it follows that both $x^2+27y^2$ and $4x^2\pm2xy+7y^2$ do not represent $3^1$. 
Then the assertion follows from Theorem \ref{MainTheorem}. 
\end{proof}

\section*{Acknowledgements}
The author would like to thank my supervisor 
Tomokazu Kashio 
for grateful advices. 
The author also thanks Master's students 
Masashi Katou, 
Yudai Tanaka 
and 
Hyuuga Yoshizaki 
for useful discussions.




\begin{thebibliography}{9}

\bibitem{Cho} B. Cho, Integers of the form $x^2+ny^2$, Monatsh Math. {\textbf{174}} (2014), 195-204.

\bibitem{Cho2} B. Cho, Representations of integers by the quadratic form $x^2+xy+ny^2$, J. Aust. Math. Soc. {\textbf{100}} (2016), 182-191. 

\bibitem{Cox} D. A. Cox, Primes of the Form $x^2+ny^2$, 2nd edition, Wiley (2013).

\bibitem{Gau} C. F. Gauss, Disquisitiones Arithmeticae, Yale, New Haven (1966).

\bibitem{Gro} E. Grosswald, Representations of Integers as Sums of Squares, Springer, Berlin (1985).

\bibitem{Ire} K. Ireland and M. Rosen, A Classical Introduction to Modern Number Theory, Springer-Verlag, Berlin, Heidelberg, and New York (1982).

\bibitem{Koo} J. K. Koo and D. H. Shin, On the Diophantine equation $pq=x^2+ny^2$, 
preprint 
(arXiv:1404.1060)
.



\end{thebibliography}
\end{document}